\documentclass[12pt,reqno]{amsart} 

\usepackage{amsmath,amssymb,rawfonts}
\usepackage{tikz}
\usepackage{mathrsfs}
\usepackage{scalerel}
\usepackage{tikz-cd}
\usetikzlibrary{cd}

\usetikzlibrary{arrows, matrix}
\usepackage{mathtools}
\usepackage{stmaryrd}

\newtheorem{theorem}{Theorem}[section]

\newtheorem{defi}[theorem]{Definition}
\newtheorem{prop}[theorem]{Proposition}
\newtheorem{coro}[theorem]{Corollary}

\newtheorem{ex}[theorem]{Example}

\title{FILTERS AND $\mathfrak{G}$-CONVERGENCE IN CATEGORIES}
\author{Joaqu\'in Luna-Torres}
\thanks{Programa de Matem\'aticas, Universidad Distrital Francisco Jos\'e de Caldas,  Bogot\'a D. C., Colombia (retired professor)}

\email{ jlunator@fundacionhaiko.org}

\begin{document} 

\begin{abstract}
In analogy with the classical theory of filters, for  fi\-nite\-ly com\-plete cat\-e\-go\-ries, we provide  the concepts of fil\-ter, $\mathfrak{G}$-neigh\-bor\-hood (short for ``Grothendieck-neigh\-bor\-hood")  and  cover-neigh\-bor\-hoodof a point,  with the aim of studying convergence, cluster point and closure of  sieves on objects  of that kind of categories.
\end{abstract}

\maketitle 

\section{ Introduction}
Convergence theory offers a versatile and effective framework for some areas of mathematics. Let us start by saying a few words about the history of this concept.

It was defined for the first time probably by Henri Cartan \cite{HC}.
Although the notion of a limit along a filter was defined in this work, in the maximal possible generality – the considered filter could be a filter on an arbitrary set and the limit was defined for any map from this set to a topological space – the attention of mathematicians in the following years was mostly focused to two special cases.

\begin{itemize}
\item In general topology the notion of limit of a filter on a topological space $X$ became one of the two basic tools used to describe the convergence in general topological spaces together with the notion of a net.
Some authors studied also the convergence of a sequence along a filter.  
 
\item The definition of the limit along a filter can be reformulated using ideals – the dual notion to
the notion of filter. This type of limit of sequences was introduced independently by P. Kostyrko et alt., \cite{KM}
 and F. Nuray and W. H. Ruckle \cite{NR} and studied under the name {\it{`` I-convergence"}}. The motivation for this direction of research was an effort to generalize some known results on statistical convergence. 
\end{itemize}
In category theory,  a sieve is a way of choosing arrows with a common codomain. It is a categorical analogue of a collection of open subsets of a fixed open set in topology. In a Grothendieck topology, certain sieves become categorical analogues of open covers in general topology. 

In this paper, we use the concept of sieve to build filters in categories and establish a link between filters and grothendieck topologies.

The paper is organized as follows: We describe, in section $2$, the notion of sieve as in S. MacLane and I. Moerdijk \cite{MM} . In section $3$, we present the concepts of filters, base of filters, we  study the lattice structure of  all filters on a category and we present the concept of ultrafilter; after that, in section $4$,  we introduce the concepts of Systems of neighborhood,$\mathfrak{G}$-neighborhood  of a point, cover-neighborhood, convergence, cluster point and closure of a sieve and some theorems about them.

\section{Theoretical Considerations}
Throughout this paper, we will work within an ambient category $\mathscr{C} $ which is finitely complete.

Remember that if $\mathscr{C} $ is a category and $C$ is an object of  $\mathscr{C} $, {\bf a sieve} 
$S: \mathscr{C}^{op} \rightarrow Set$ on $C$ is a subfunctor of $Hom_{\mathscr{C}} (-,C)$

If we denote by  $Sieve_{\mathscr{C}}(C)$ (or $Sieve(C)$ for short)  the set of all sieves on $C$, then $Sieve(C)$ becomes a partially ordered under inclusion. It is easy to see that the union or intersection of any family of sieves on $C$ is a sieve on $C$, so $Sieve(C)$ is a complete lattice.

From S. MacLane and I. Moerdijk \cite{MM}, Chapter III, we have the following:
\begin{itemize}
\item Let $\mathscr{C} $ be a category, and let $Sets^{\mathscr{C}^{op}}$  be the corresponding functor category and let 
\begin{align*}
y: &\mathscr{C} \rightarrow  Sets^{\mathscr{C}^{op}}\\
&C\mapsto \mathscr{C}(-,C)
\end{align*}
then a sieve $\mathcal S$ on $\mathscr{C}$ is simply a subobject  $\mathcal S\subseteq y(C)$ in $Sets^{\mathscr{C}^{op}}$.

Alternatively, $\mathcal S$ may be given as a family of morphisms in $\mathscr{C}$, all with codomain $C$, such that
\[
f \in  \mathcal S \Longrightarrow f\circ g \in \mathcal S
\]
whenever this composition makes sense; in other words, $\mathcal S$ is a right ideal. 

\item If $\mathcal S$ is a sieve on $\mathscr C$ and $h: D\rightarrow C$ is any arrow to $C$, then  $h^{*}(\mathcal S) = \{g \mid cod(g) = D, h\circ g \in \mathcal S\}$  is a sieve on $D$.
\end{itemize}

\section{Filters on a category}
\begin{defi}
 A filter on a category $\mathscr{C}$ is a function $\mathfrak F$ which assigns to each object $C$ of $\mathscr{C}$ a  collection $\mathfrak F(C) $ of sieves on $\mathscr{C}$, in such a way that
\begin{enumerate}
\item [($F_1$)] If $S \in \mathfrak F(C)$ and $ R$ is a sieve on $C$ such that $ S \subseteq R$, then  $ R \in \mathfrak F(C)$;
\item [($F_2$)] every finite intersection of sieves of $\mathfrak F(C)$ belongs to $\mathfrak F(C)$;
\item [($F_3$)] the empty sieve is not in $\mathfrak F(C)$.
\end{enumerate}
\end{defi}
The pair $(C, \mathfrak F(C))$ is called {\bf a filtered object}.
\begin{ex}
From the definition of a Grothendieck topology $J$ on a category $\mathscr{C}$  it follows that for each object $C$ of $\mathscr{C}$ and that
\begin{itemize}
\item for $S\in J(C)$ any larger sieve $R$ on $C$ is also a member of $J(C)$. 
\item It is also clear that if $R; S\in J(C)$ then $R\cap S \in J(C)$, 
\item consequently a Grothendieck topology produces a filter in the same category $\mathscr{C}$; it is enough to remove the empty sieve of $J(C)$. 
\end{itemize}
\end{ex}
\begin{defi}
A filter subbase on a category $\mathscr{C}$ is a function $\mathfrak S$ which assigns to each object $C$ of $\mathscr{C}$ a  collection $\mathfrak S(C) $ of sieves on $\mathscr{C}$,  in such a way that no finite subcollection of $\mathfrak S(C) $ has an empty intersection.
\end{defi}
An immediate consequence of this definition is
\begin{prop}
A sufficient condition that there should exist a filter $\mathfrak S^{'}$  on a category $\mathscr{C}$ greater than or equal to  a function $\mathfrak S$ (as above) is that $\mathscr{C}$ should be a filter subbase on  $\mathscr{C}$.
\end{prop}
Observe that $\mathfrak S^{'}$ is the coarset filter greater than $\mathfrak S$.
\begin{defi}
A basis of a filter on a category $\mathscr{C}$ is a function $\mathfrak B$ which assigns to each object $C$ of $\mathscr{C}$ and   collection $\mathfrak B(C) $ of sieves on $\mathscr{C} $, in such a way that 
\begin{enumerate}
\item[($B_1)$] The intersection of two sieves of $\mathfrak B(C)$ contains a sieve of $\mathfrak B(C)$;
\item[$(B_2)$] $\mathfrak B(C)$ is not empty, and the empty sieve is not in $\mathfrak B(C)$.
\end{enumerate}
\end{defi}
\begin{prop}
If $\mathfrak B$ is a basis of filter on a category $\mathscr{C} $, then  $\mathfrak B$ generates a filter $\mathfrak F$ by
\[
S \in \mathfrak F(C) \Leftrightarrow \exists R\in \mathfrak B(C)\,\ \text{such that}\,\ R\subseteq S
\]
for each object $C$ of $\mathscr{C} $.
\end{prop}

It is easy to check that this, indeed, defines a filter from a basis $\mathfrak B$.

\subsection{The ordered set of all filters on a category}
\begin{defi}
Given two filters $\mathfrak F_1$,\,\ $\mathfrak F_2$ on the same category $\mathscr{C}$,  $\mathfrak F_2$ is said to be finer than $\mathfrak F_1$, or $\mathfrak F_1$ is coarser than  $\mathfrak F_2$, if  $$\mathfrak F_1(C) \subseteq  \mathfrak F_2(C)$$ for all  $C$  object of  $\mathscr{C}$.
\end{defi}
In this way, the set of all filters on a category $\mathscr{C}$ is ordered by the relation {\bf{ ``$\mathfrak F_1$ is coarser than  $\mathfrak F_2$"}}.

Let $(\mathfrak F_i)_{i\in I}$ be a nonempty family of filters on a category $\mathscr{C}$; then the function $\mathfrak F$ which assigns to each object $C$ the collection $$\mathfrak F(C)=\bigcap_{i\in I}\mathfrak F_i(C)$$ is manifestly a filter on  $\mathscr{C}$ and is obviously the greatest lower bound of the family $(\mathfrak F_i)_{i\in I}$ on the ordered set of all filters on   $\mathscr{C}$.
\begin{defi}
An ultrafilter on  a category $\mathscr{C}$ is a filter $\mathfrak U$ such that there is no filter on $\mathscr{C}$ which is strictly finer than $\mathfrak U$..
\end{defi}
Using the Zorn lemma, we deduce that
\begin{prop}
If $\mathfrak F$ is any filter on a category $\mathscr{C}$,  there is an ultrafilter finer than $\mathfrak F$ on $\mathscr{C}$.
\end{prop}
\begin{prop}
Let $\mathfrak{U}$ be an ultrafilter on a category $\mathscr{C}$, and let $C$ be an object of $\mathscr{C}$. Let $ S,T$ be sieves on $C)$ such that $S\cup T \in \mathfrak{U}(C)$ then either $S \in \mathfrak{U}(C)$ or $T \in \mathfrak{U}(C)$.
\end{prop}
\begin{proof}
If the affirmation is false, there exist sieves $ S,T$ on $C$ that do not belong to $ \mathfrak{U}(C)$, but $S\cup T \in \mathfrak{U}(C)$.
Consider a function  $\mathfrak{T}:\mathscr{C}\rightarrow Set$  and let us build  the collection of sieves $R$ on $C$ such that $ S\cup R\in \mathfrak{U}(C)$.
Let us verify that $\mathfrak{T}$ is a filter on a  $\mathscr{C}$: Let  $C$ be an object of the category $\mathscr{C}$, and let $\mathfrak{T}(C)$ be the image of $C$  under $\mathfrak{T}$; in other words, $\mathfrak{T}(C)=\{ R\in \,\ Sieve(C)\mid R\cup S \in \mathfrak{U}(C)\}$, then 
\begin{enumerate}
\item [($F_1$)] if $R^{'} \in \mathfrak T(C)$ then $R^{'} \cup S\in \mathfrak{U}(C)$; and if  $ R^{''}$ is a sieve on $C$ such that $ R{'}\subseteq R{''}$, then  $ R^{''}\cup S \in \mathfrak U(C)$. Consequently $R^{''} \in \mathfrak T(C)$.
\item [($F_2$)] We must show that every finite intersection of sieves of $\mathfrak T(C)$ belongs to $\mathfrak T(C)$; indeed,
 let 
\[
(R_i)_{i=1,\cdots,n}
\]
 be a finite collection of sieves on $C$ such that 
\[
R_i\cup S\in \mathfrak{U}(C),\,\ \text{for all}\,\  i = 1. . . n,
\] 
then
\[
(R_1\cup S)\cap (R_2\cup S)\cap \cdots \cap (R_n\cup S) = \left(\bigcap_{i=1}^{n}R_i\right)\cup S\in \mathfrak{U}(C).
\]
which is equivalent to saying that 
\[
\left(\bigcap_{i=1}^{n}R_i\right) \in \mathfrak T(C).
\]
\item [($F_3$)] Evidently, the empty sieve is not in $\mathfrak T(C)$.
 \end{enumerate}
Therefore $\mathfrak{T}$ is a filter finer than $\mathfrak{U}$, since $T\in \mathfrak T(C)$; but this contradicts the hypothesis than 
$\mathfrak{U}$ is an ultrafilter.
\end{proof}
\begin{coro}
If the union  of a finite sequence $(S_i)_{i=1,\cdots,n}$ of sieves on $C$ belongs to the image, $\mathfrak{U}(C)$,  of an object $C$ under  an ultrafilter $\mathfrak{U}$, then at least one of the $S_i$ belongs to $\mathfrak{U}(C)$.
\end{coro}
\begin{proof}
The proof is a simple use of induction on $n$.
\end{proof}

\section{Systems of Neighborhoods}

Remember that a {\bf point} of  an object $C$ of a category $\mathscr{C}$ is a morphism $p:1\rightarrow C$, where $1$ is a terminal object of $\mathscr{C}$.

\begin{defi}
Let  $(\mathscr C, J)$ be a category endowed with a Grothendieck topology, and let $C$  be an object  of $\mathscr{C}$.
\begin{enumerate}
\item A sieve $V$ in $J(C)$,  is said to be a  $\mathfrak{G}$-neighborhood (short for ``Grothendieck-neighborhood") of a point $p_{\scriptscriptstyle C}$ of $C$ if there exist a morphism  $\phi:D\rightarrow C$ in  $V$ and  a point $q:1 \rightarrow D $ such that the diagram

\begin{center}
\begin{tikzcd}
D \arrow{r}{\phi} & C \\ 
  1\arrow{u}{q}\arrow{ur}[swap]{p}& 
\end{tikzcd}
\end{center}
is commutative.
\item A sieve $V$ in $J(C)$,  is said to be a  $\mathfrak{G}$-neighborhood of a sieve $T$ on $C$ if $T\subseteq V$.

\end{enumerate}
\end{defi}

\begin{defi}
Let  $(\mathscr C, J)$ be a category endowed with a Grothendieck topology. A cover-neighborhood of $(\mathscr C, J)$ is a function $\mathcal N$ which assigns to each object $(C, J(C))$ of $(\mathscr{C}, J)$ and to each point $p_{\scriptscriptstyle C}$ of $C$, a collection 
\[
\mathcal N_{\scriptstyle p_{_{\scriptscriptstyle C}}}(C)\,\ \text{of sieves of}\,\ \mathscr C
\]
such that each sieve in
 $\mathcal N_{\scriptstyle p_{_{\scriptscriptstyle C}}}(C)$ contains a $\mathfrak{G}$-neighborhood of $p_{\scriptscriptstyle C}$.
\end{defi}
%\begin{defi}
%Let $\mathscr{C} $ be a finitely complete category. We define a pointed category $\mathscr{C}^{*}$ as follows. Its objects are pairs $(C,p_c)$, where $C$ is an object of $\mathscr{C}$ and $p_c$ is a point of $C$. Morphisms $\phi: (D,p_d)\rightarrow(C.p_c)$ are morphisms of $\mathscr{C}$ such that $\phi\circ p_d =p_c$.
%\end{defi}

\begin{theorem}
Let $\mathscr{C}$ be  a  category, and let $C$ be an object of $\mathscr{C}$. The pair $(C,\mathscr N_p(C))$, where  $\mathscr N_p(C)$ is the collection of all  cover-neighborhoods of a point $p$ of $C$,  is  {\bf a filtered object}
\end{theorem}
\begin{proof}\
\begin{enumerate}

\item [(i)] If $S\in\mathscr N_p(C)$ and $R$ is a sieve on $C$ such that $S \subseteq R$, then  $R\in\mathscr N_p(C)$, because there is a $\mathfrak{G}$-neighborhood $V$ of $p_{\scriptscriptstyle C}$ such that $V \subseteq S\subseteq R$;
\item[(ii)] let $\{S_1, S_2,\cdots,S_n\}$ be a finite col\-lec\-tion of sie\-ves of $\mathscr N_p(C)$, then there exists a col\-lec\-tion    $\{V_1, V_2,\cdots, V_n\}$  of $\mathfrak{G}$-neigh\-bor\-hood of $p_{\scriptscriptstyle C}$ such that $V_i\subseteq S_i$ for $I=1,2,\cdots n$, therefore
$$\bigcap_{i=1}^n V_n \subseteq\bigcap_{i=1}^n S_n \,\ \text{and}\,\ \bigcap_{i=1}^n S_n\in\in\mathscr N_p(C)$$;
\item[(iii)] the empty sieve is not in $\mathscr N_p(C)$ (each  sieve contains a point).
\end{enumerate}
\end{proof}
In this case, we say that the point $p$ of $C$ is a {\bf limit point} of $\mathscr N_p(C)$. 

\begin{defi}\label{converges}
Let  $(\mathscr C, J)$ be a  category endowed with a Grothendieck topology; let  $\mathfrak F$ be a filter on  $\mathscr{C}$ and let $C$  be an object  of $\mathscr{C}$. 
\begin{enumerate}
\item We shall say that $\mathfrak F(C) $ {\bf converges} to a point $p$ of $C$ if\linebreak $\mathscr N_p(C)\subseteq \mathfrak F(C) $.
\item The closure of a sieve $A$ on $C$ is the collection of all points $p$ of $C$ such that every cover-neighborhood of $p$ meets $A$.
\item A point $p$ of $C$ is a cluster point of $\mathfrak B(C)$, the image under the filter base $\mathfrak B$ of $C$, if it lies in the closure of all the sieves on $\mathfrak B(C)$.
\item A point $p$ of $C$ is a cluster point of $\mathfrak F(C)$, the image under the filter  $\mathfrak F$ of $C$, if it lies in the closure  of all the sieves on $\mathfrak F(C)$.
\end{enumerate}
\end{defi}
\begin{theorem}
Let  $(\mathscr C, J)$ be a  category endowed with a Grothendieck topology; let  $\mathfrak F$ be a filter on  $\mathscr{C}$ and let $C$  be an object  of $\mathscr{C}$. The point $p$ of $C$ is a cluster point of  $\mathfrak F(C)$ iif there is a filter  $\mathscr G$ finer than  $\mathfrak F$ such that $\mathscr G(C) $ {\bf converges} to $p$.
\end{theorem}
\begin{proof}
Let us begin by assuming that the point $p$ of $C$ is a cluster point of  $\mathfrak F(C)$; from definition \ref{converges}, it follows that for each sieve $A$ in $\mathfrak F(C)$, every $\mathfrak{G}$-neighborhood $V$ of $p$ meets $A$. We need to show that the collection 
\[
\mathscr B(C) =\{A\cap V\mid V \,\ \text{is a $\mathfrak{G}$-neighborhood of}\,\ p\}
\]
define  a base for a filter  $\mathscr G$ finer than $\mathfrak F$.in such a way that  $\mathscr G(C) $ {\bf converges} to $p$.

Indeed,
\begin{enumerate}
\item[($B_1)$] Let $A\cap V$,\,\ $A\cup W$ two elements of the collection $\mathscr B(C)$, since 
$$(A\cap V) \cap (A\cup W)= A\cup (V\cap W) $$
and $V\cap W$ is a $\mathfrak{G}$-neighborhood of $p$, there exists $U$,  a $\mathfrak{G}$-neighborhood of $p$ such that
$$U\subseteq V\cap W,$$ and clearly $A\cap U \in \mathscr B(C)$;
\item[$(B_2)$] Obviously $\mathscr B(C)$ is not empty, and the empty sieve is not in $\mathscr B(C)$.
\end{enumerate}
Now, if $\mathscr G$ is the filter generated by $\mathscr B$ then  $\mathscr G$ is finer than  $\mathfrak F$, and $\mathscr G(C)$ naturally converges to $p$.

Conversely, if there is a filter  $\mathscr G$ finer than  $\mathfrak F$ such that $\mathscr G(C) $ {\bf converges} to $p$ then each sieve $R$ in $\mathfrak F(C)$ and each $\mathfrak{G}$-neighborhood $U$ of $p$ belongs to  $\mathscr G$ and hence meet, so the point $p$ of $C$ is a cluster point of  $\mathfrak F(C)$.
\end{proof}
\begin{theorem}
Let  $(\mathscr C, J)$ be a  category endowed with a Grothendieck topology; let  $C$  be an object  of $\mathscr{C}$ and let $A$ be a sieve on  $C$. The point $p$ of $C$ lies in the closure  of $A$ iif there is a filter  $\mathscr G $ such that  $A\in\mathscr G(C) $ and $\mathscr G(C) $ {\bf converges} to $p$.
\end{theorem}
\begin{proof}
Let us begin by assuming that The point $p$ of $C$ lies in the closure  of $A$; from definition \ref{converges}, it follows that  every $\mathfrak{G}$-neighborhood $V$ of $p$ meets $A$. Then 
\[
\mathscr B(C) =\{A\cap V\mid V \,\ \text{is a $\mathfrak{G}$-neighborhood of}\,\ p\}
\]
is a base  for a filter  $\mathscr G$, in such a way that  $\mathscr G(C) $ {\bf converges} to $p$.

Conversely, if  $A \in \mathscr G(C)$ and  $\mathscr G(C) $ {\bf converges} to $p$ then  $p$ is a cluster point  of $\mathscr G(C) $ and hence $p$ lies in the closure  of $A$.
\end{proof}

\begin{ex}
Let $\mathcal A$ be a complete Heyting algebra and regard  $\mathcal A$  as a category in the usual way. 
\begin{itemize}
\item \cite{MM} Then  $\mathcal A$ can be equipped with a base for a Grothendieck topology $K$, given by
\[
\{a_i\mid i\in I\}\in K(c)\,\ \text{if and only if}\,\ \bigvee_{i\in I}=c,
\]
where $\{a_i\mid i\in I\}\subseteq \mathcal A$ and $c\in \mathcal A$.

\item A sieve $S$ on an element $c$ of  $\mathcal A$ is just a subset of elements $b\leqslant c$ such that $a\leqslant b\in S$ implies $a\in S$.
\item In the Grothendieck topology $J$ with basis $K$, a sieve $S$ on $c$ covers $c$  iff
\[
\bigvee S=c.
\]
\item A filter on t$\mathcal A$ is a function $\mathfrak F$ which assigns to each element $c$ of $\mathcal A$ a  collection $\mathfrak F(c) $ of sieves, such that 
\begin{enumerate}
\item [($F_1$)] If $S \in \mathfrak F(c)$ and $ R$ is a sieve on $c$ such that $ S \subseteq R$, then  $ R \in \mathfrak F(c)$;
\item [($F_2$)] every finite intersection of sieves of $\mathfrak F(c)$ belongs to $\mathfrak F(c)$;
\item [($F_3$)] the empty sieve is not in $\mathfrak F(c)$.
\end{enumerate}
\item An immediate consequence of the previous construction of  a Grothendieck topology and a fiter on $\mathcal A$ is that 

\[
\mathfrak F(c)\,\ \text{converges to $c$ iff}\,\   \bigvee S=c, \,\ \text{for each}\,\  S \in \mathfrak F(c).
\]
\end{itemize}
\end{ex}
\section*{Conclusions and Comments}
We proposed the construction of the concepts of filter, $\mathfrak{G}$-neighbor\-hood of a point and  cover-neighborhood with the aim of studying convergence, cluster point and closure of  sieves on objects of some kind of categories.
The properties  of these objects (similar to those general topology) are  postponed to future works.
\section*{Acknowledgments}.
 The author wish to thank \emph{ Fundaci\'on Haiko} of Colombia for its constant encouragement during the development of this work.

\end{document}